\documentclass{amsart}

\usepackage{amsmath,amssymb,amsthm}
\usepackage{graphicx,tikz}
\newtheorem{theorem}{Theorem}
\newtheorem*{thm}{Theorem}

\newtheorem{lemma}{Lemma}

\theoremstyle{definition}

\theoremstyle{remark}

\begin{document}

\title[]{Hot Spots in convex domains are \\in the tips (up to an inradius)}
\thanks{S.S. is supported by the NSF (DMS-1763179) and the Alfred P. Sloan Foundation.}

\author[]{Stefan Steinerberger}
\address{Department of Mathematics, Yale University, New Haven, CT 06511, USA}
\email{stefan.steinerberger@yale.edu}

\begin{abstract}
Let $\Omega \subset \mathbb{R}^2$ be a bounded, convex domain and let $-\Delta \phi_1 = \mu_1 \phi_1$ be the first
nontrivial Laplacian eigenfunction with Neumann boundary conditions. The Hot Spots conjecture claims that
the maximum and minimum are attained at the boundary. We show that they are attained far away from one another: if $x_1, x_2 \in \Omega$ satisfy 
$\|x_1 - x_2\| = \mbox{diam}(\Omega)$, then every maximum and minimum is
assumed within distance $c\cdot \mbox{inrad}(\Omega)$ of $x_1$ and $x_2$, where $c$ is a universal constant (which is the optimal scaling up to the value of $c$).
\end{abstract}

\maketitle

\section{Introduction and result}

\subsection{Introduction.} Let $\Omega \subset \mathbb{R}^2$ be a convex domain and let $\phi_1$ denote the first
nontrivial Laplacian eigenfunction with Neumann boundary conditions, i.e.
\begin{align*}
-\Delta \phi_1 &= \mu_1 \phi_1 \quad \mbox{in}~ \Omega\\
\frac{\partial \phi_1}{\partial \nu} &= 0 \quad \mbox{on}~\partial \Omega
\end{align*}
Physically, this describes the long-term behavior of generic solutions of the heat equation if $\Omega$ is insulated -- physically, one would expect that the maximum and the minimum are at the boundary, this is the Hot Spots conjecture of J. Rauch.

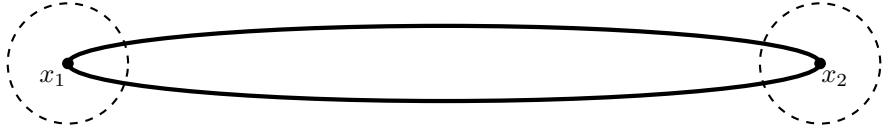
\begin{figure}[h!]
\begin{tikzpicture}
\draw[ultra thick] (0,0) ellipse (5cm and 0.5cm);
\filldraw (-5,0) circle (0.07cm);
\node at (-5.2, -0.2) {$x_1$};
\filldraw (5,0) circle (0.07cm);
\node at (5.2, -0.2) {$x_2$};
\draw [thick, dashed] (-5,0) circle (0.8cm);
\draw [thick, dashed] (5,0) circle (0.8cm);
\end{tikzpicture}
\caption{Maximum and minimum are attained close (at most a universal multiple of the inradius away) to the points achieving maximal distance (the `tips' of the domain).}
\end{figure}

However, physical intuition also tells us that the maxima and minima should not be merely at the boundary, they should be at `opposite ends' of the domain (see Fig. 1). The purpose of this paper is to prove a sharp form of this statement.
\begin{thm} There exists a universal $c>0$ such that for all bounded, convex $\Omega \subset \mathbb{R}^2$: if $x_1, x_2 \in \Omega$ are at maximal distance, $\|x_1 -x_2\| = \emph{diam}(\Omega)$, then $\phi_1$ assumes \emph{every} global maximum and minimum at distance at most $c \cdot \emph{inrad}(\Omega)$ from $\left\{x_1, x_2\right\}$.
\end{thm}

This result may be understood as being somewhat parallel to the Hot Spots conjecture. Dilation of the domain shows that it is sharp up to the value of the constant $c$. Our proof can, in principle, produce an effective upper bound on $c$ but it is unlikely to be close to the true value. The example of the disk shows that the optimal constant has to satisfy $c \geq \sqrt{2}$. The example of an $N \times 1$ rectangle shows that, even for very eccentric convex sets, one has $c \geq 1$. 

\subsection{Related results} The Hot Spots conjecture  was  made by Rauch in 1974 in a lecture he gave at a Tulane University PDE conference \cite{b3}. Except for it being mentioned in a 1985 book of Kawohl \cite{kawohl} (who suggested that it may be false in general but might be true for convex domain), there was relatively little work until the late 1990s.
Then Ba\~nuelos \& Burdzy \cite{b3} proved it for obtuse triangles
and Burdzy \& Werner \cite{b4} obtained a counterexample for domains that are not simply connected (see also Burdzy \cite{b25} for a counterexample with one hole). Jerison \& Nadirashvili \cite{jerison} proved it for convex domains with two axes of symmetry, Pascu \cite{pascu} proved it for convex domains with one axis of symmetry. Atar \& Burdzy \cite{atar} proved it for a family of convex domains, where, essentially, the boundary is given by 1-lipschitz functions (their result is actually stronger and gives detailed information about the level sets in that case). Miyamoto \cite{miya} proved the conjecture for convex domains `close' to the ball and constructed an example \cite{miya2} showing that there might be arbitrarily many isolated hot spots on the boundary (the rectangle shows that there can be infinitely many non-isolated hot spots). The case of triangles was studied by the Polymath7 project, Judge \& Mondal \cite{judge} recently established the Hot Spots conjecture for all triangles. 
Our result is reminiscent in flavor of a result of Jerison \cite{jerison} who located the position of the nodal line (up to an inradius). We also note that the problem with Dirichlet instead of Neumann boundary conditions is fairly well understood \cite{beck, brasco, grieser2, magna, manas}. \\

Many of the results for the Neumann problem \cite{atar, b25, b3, b4, pascu} are established using probabilistic techniques which could be rewritten in purely analytic terms (`markovianity' = `semigroup property', `Brownian motion' = `heat kernel') but can be explained particularly nicely in probabilistic language; we will also give a probabilistic argument. Our argument has quite a bit of flexibility and can be used in much rougher settings (say, the Graph Laplacian on combinatorial graphs, general second-order uniformly elliptic operators on suitable manifolds, ...) but the formulation of the Theorem as well as the proof are particularly nice and transparent in the case of convex domains in the plane.

\section{Proof}
\subsection{Outline.} The proof combines two-sided Gaussian bounds for the heat kernel (cf. Grigor'yan \cite{grig}, Saloff-Cose \cite{saloff}, Sturm \cite{sturm}) with combinatorial reasoning on top of the probabilistic interpretation that inspired most of the existing results (cf. Ba\~nuelos \& Burdzy \cite{atar} and Burdzy \& Werner \cite{b4}). 
 We rotate $\Omega$ so as to minimize the projection onto the $y-$axis (see Fig. 2) and use dilation to normalize to have inradius $\mbox{inrad}(\Omega) = 1$. We can thus assume to be dealing with the case where $\Omega$ has dimensions roughly $N \times 1$, where $N \gg 1$. In particular, our result allows us to assume w.l.o.g. that $N \geq 100$ and a result of Ba\~nuelos \& Burdzy \cite[Proposition 1.1]{b3} then shows that the first eigenvalue is simple (allowing us to speak of \textit{the} first eigenfunction).
 A result of Jerison \& Grieser tells us that the nodal line is very close to a line and has width (projection onto the $x-$axis) $\leq c N^{-1}$ (we do not need this statement but it helps to paint the picture).
We give the proof only for the maximum, the case of the minimum is identical after multiplying with $(-1)$. We also assume, for simplicity of exposition, that the maximum is assumed `on the right' in the normalization of Fig. 2 (again without loss of generality). We show the following: if $\phi_1$ assumes its maximum in $(x,y) \in \Omega$,
then, for some universal constant $c$ and all $z\in \mathbb{R}$, either $(x+c, z) \notin \Omega$ or $\phi_1(x+c, z) > \phi_1(x,y)$.
The proof comes in two parts.
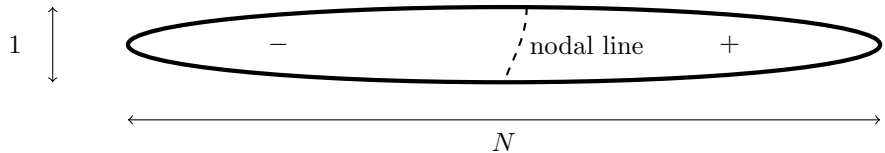
\begin{figure}[h!]
\begin{tikzpicture}
\draw[ultra thick] (0,0) ellipse (5cm and 0.5cm);
\draw [<->] (-5,-1) -- (5,-1);
\node at (0, -1.3) {$N$};
\draw [<->] (-6,-0.5) -- (-6,0.5);
\node at (-6.5, 0) {$1$};
\draw [thick, dashed] (0.3, 0.5) to[out=270,in=70] (0, -0.5);
\node at (1.1, 0) {$\mbox{nodal line}$};
\node at (3,0) {$+$};
\node at (-3,0) {$-$};
\end{tikzpicture}
\caption{The normalization of $\Omega$.}
\end{figure}

\textit{Part 1.}
 The bulk of the work goes into the following statement: if $\phi_1$ assumes its maximum in $x_{\max}$, then in a 1-neighborhood of $x_{\max}$, i.e. for any $\|x-x_{\max}\| \leq 1$,
$$\phi_1(x) \geq   e^{-c \cdot \mu_1} \phi_1(x_{\max})$$
where $c$ is a universal constant. We note that, for another universal constant $c$, we have $\mu_1 \leq c N^{-2}$, the exponential factor $e^{-c \cdot \mu_1}$ is thus very close to 1 and at least of size $1 - c N^{-2}$.
This scaling is rather unsurprising -- it can be summarized as saying that, locally, Taylor expansion is accurate. We recall that, locally in $x_{\max}$, the gradient vanishes, $\nabla \phi_1(x_{\max}) = 0$, and Taylor expansion results in
$$ \phi_1(x) = \phi_1(x_{\max}) + \left\langle D^2 \phi_1(x_{\max}) (x-x_{\max}), x-x_{\max} \right\rangle + \mbox{l.o.t.}$$
The eigenvalue satisfies the asymptotic bound $\mu_1 \leq c N^{-2},$ where $c$ is a universal constant (see \cite{grieser, jerison}). Moreover
$$ \mbox{tr} ~D^2 \phi_1 = \Delta \phi_1 = -\mu_1 \phi_1$$
and the Hessian is negativ-semidefinite in the maximum, therefore
$$ \left\langle D^2 \phi_1(x_{\max}) v,v \right\rangle \geq -\mu_1 \phi_1(x_{\max}) \|v\|^2$$
and we expect, from Taylor expansion,
$$ \phi_1(x) \geq \phi_1(x_{\max}) \left(1- \mu_1 \|x-x_{\max}\|^2\right) \geq (1-\mu_1) \phi_1(x_{\max}).$$
We show that this argument can be extended from the infinitesimal scale to scale 1. We remark that this inequality also implies a geometric version: for $\|x - x_{\max}\| \leq 1$ and a universal constant $c>0$
 $$\phi_1(x) \geq   \left(1 - c \left( \frac{\mbox{inrad}(\Omega)}{\mbox{diam}(\Omega)}\right)^2\right)\phi_1(x_{\max}).$$
 It is clear that such a lower bound cannot be true on general domains (imagine bottleneck domains with a very thin bottleneck) but it should universally be true for convex domains (and that is what we prove in $d=2$ dimensions). It is conceivable that this part of our proof can be simplified. \\

\textit{Part 2.} The second part of the proof makes use of the probabilistic interpretation
$$ \phi_1(x) = e^{\mu_1 t} \cdot \mathbb{E}(\phi_1(B_x(t))),$$
where $B_x(t)$ is a Brownian motion started in $x$ and running up to time $t$ that is being reflected on the boundary. The `averaged' case corresponds to the heat kernel but we will actually make explicit use of the internal time $t$ of a Brownian motion.
The second part of our argument is conceptually robust and might potentially be useful in establishing even stronger results or similar results in different settings. The critical bottleneck that is missing to establish, say, the hot spots conjecture seems to be an isoperimetric principle which we outline in the last section of the paper. It is clear that one of the difficulties of the hot spots conjecture is that the location of the maximum of a function is, by definition, unstable (a small perturbation, say size $\varepsilon$ in $L^{\infty}$, can move the maximum by $\gtrsim \sqrt{\varepsilon}$ and even further if the critical point is unstable) -- one nice aspect of our approach is that it is local and exploits the structure of the maximum in a very crucial way (the simplest description of the approach is the old adage that if we take an average over many numbers and the average is close to the maximum, then most numbers most be close to the maximum; indeed, our approach would not work as easily away from the maximum); we are somewhat hopeful that it can find other applications.

\subsection{A Geometric Lemma} We establish several Lemmata. The first says that it does not matter which pair of extremizing points we chose. As such, it will not be important for the main part of the proof of the Theorem but it is important to establish the internal validity of the statement.
\begin{lemma} There exists a universal constant $c>0$ such that if
$$ \| x_1 - x_2\| = \emph{diam}(\Omega) = \|x_3 - x_4\|,$$
then, up to possibly renaming the variables, $\|x_1 - x_3\| \leq c\cdot \emph{inrad}(\Omega)$.
\end{lemma}
\begin{proof}
Since we are only interested in the existence of a constant, we can assume w.l.o.g. $\mbox{diam}(\Omega) \geq 1000$ and $\mbox{inrad}(\Omega) = 1$. The points $x_1$ and $x_2$ define, since they are at maximum distance, two circular arcs that bound the possible region in which $x_3, x_4$ (along with the rest of $\Omega$) can be (see Fig. 3). We argue via contradiction: suppose $\|x_3 - x_1\| \geq 100$ and $\|x_3 - x_2\| \geq 100$.

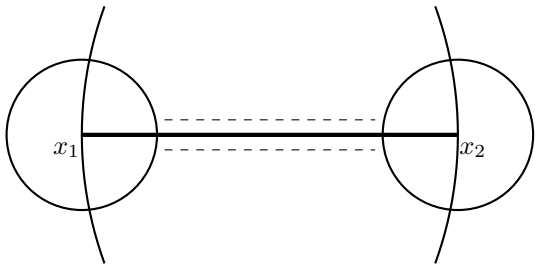
\begin{figure}[h!]
\begin{tikzpicture}
\draw[ultra thick] (0,0) -- (5,0);
\node at (-0.2, -0.2) {$x_1$};
\node at (5.2, -0.2) {$x_2$};
\draw [thick] (0,0) circle (1cm);
\draw [thick] (5,0) circle (1cm);
   \draw [thick,domain=160:200] plot ({5+5*cos(\x)}, {5*sin(\x)});
      \draw [thick,domain=-20:20] plot ({5*cos(\x)}, {5*sin(\x)});
      \draw[dashed] (1.1, 0.2) -- (3.9, 0.2);
            \draw[dashed] (1.1, -0.2) -- (3.9, -0.2);
\end{tikzpicture}
\caption{Proof of the Lemma.}
\end{figure}

We note this severely restricts where $x_3$ can be. In particular, since the triangle $x_1 x_2 x_3$ is in $\Omega$ and $\Omega$ has $\mbox{inrad}(\Omega) = 1$, this means that $x_3$ has to be very close to the line $x_1 x_2$ (in the dashed region whose height is less than 5). But then, since our consideration of $x_3$ was arbitrary, the same as to hold for $x_4$ which also has to be in the dashed region. However, the dashed region has a strictly smaller diameter, this implies $\| x_3 - x_4\| < \mbox{diam}(\Omega)$ and is a contradiction.
\end{proof}

\subsection{A Heat Kernel Estimate.} We will also use a heat kernel estimate; it is not novel but its formulation somewhat differs from the formulation commonly encountered since we require that the constants are independent of the domain.
\begin{theorem}[cf. Grigor'yan \cite{grig}, Saloff-Coste \cite{saloff2}, Sturm \cite{sturm}] There exist four universal constants $c_1, c_2, c_3, c_4 > 0$ such that for all smooth convex domains $\Omega \subset \mathbb{R}^2$ the (Neumann) heat kernel  $p_t(x,y)$ satisfies
$$ \frac{c_1}{V(x, \sqrt{t})} e^{-c_2 \frac{\|x-y\|^2}{t}} \leq p_t(x,y) \leq \frac{c_3}{V(x, \sqrt{t})} e^{-c_4 \frac{\|x-y\|^2}{t}},$$
where $V(x,r) = \left|\left\{y \in \Omega: \|x-y\| \leq r\right\}\right|$.
\end{theorem}

This result does not seem to be stated in the literature as such; it seems customary to explicitly fix the domain $\Omega$ and then allow the constants to depend on the domain (in which case, the result for convex $\Omega \subset \mathbb{R}^d$ is, for example, mentioned in \cite[Section 3.3]{saloff}). However, the result is also not new: seminal work of Grigor'yan \cite{grig}, Saloff-Coste \cite{saloff2} and Sturm \cite{sturm} implies, at a very broad level of generality, an equivalence between the parabolic Harnack inequality, the two-sided Gaussian bounds for the heat kernel that we seek, and the conjunction of
\begin{enumerate}
\item the volume doubling property $V(x,2r) \leq D V(x,r)$ and
\item the Poincar\'{e} inequality 
$$ \int_{B(x,r)}|f-f_B|^2 dx \leq P r^2 \int_{B(x,r)} |\nabla f|^2 dx,$$
where $f_B$ is the mean value of $f$ over $B(x,r)$.
\end{enumerate}
It is easy to see that the volume doubling property is satisfied with a uniform constant $D$ for convex domains in $\mathbb{R}^2$. As for the Poincar\'{e} inequality, there are uniform estimates for convex domains (and thus the constant $P$) that imply the desired inequality (see e.g. \cite{payne, stein}). We also refer to the nice presentations by Bernicot, Coulhon \& Frey \cite{bern} and in the book of Zhang \cite{zhang}. An interesting aspect of our argument is that we do not actually need very precise estimates because we work at time scale $\sim 1$ and can compensate for universal constants: in particular, \textit{any} $c_1, c_2, c_3, c_4 > 0$ in Theorem 1 would suffice for our purpose.

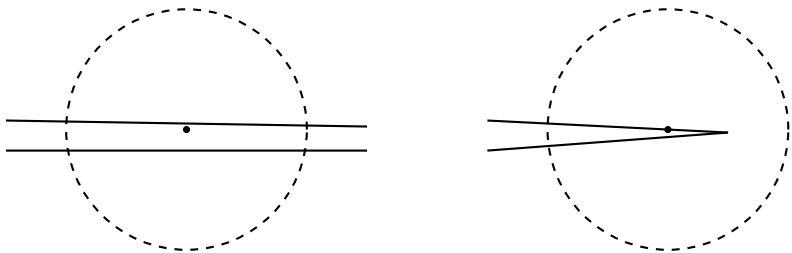
\begin{figure}[h!]
\begin{center}
\begin{tikzpicture}[scale=0.8]
\draw [thick] (0,0) -- (6,0);
\draw [thick] (0,0.5) -- (6,0.4);
\filldraw (3, 0.35) circle (0.05cm);
\draw [dashed, thick] (3, 0.35) circle (2cm);
\draw [thick] (8,0) -- (12,0.3);
\draw [thick] (8,0.5) -- (12,0.3);
\filldraw (11, 0.35) circle (0.05cm);
\draw [dashed, thick] (11, 0.35) circle (2cm);
\end{tikzpicture}
\end{center}
\caption{The two cases in the proof of Lemma 2.}
\end{figure}

We conclude this with another uniform estimate for convex domains.
\begin{lemma} For any $\delta > 0$, there exists universal $c_{1,\delta}, c_{2,\delta}$ such that for all bounded, convex domains $\Omega \subset \mathbb{R}^2$ of the type under consideration (of diameter $\emph{diam}(\Omega) \gg 1 = \emph{inrad}(\Omega)$) and all $\|x - y\|\leq1$
$$  c_{1,\delta} V(y,\delta) \leq V(x,1) \leq c_{2,\delta} V(y,\delta).$$
\end{lemma}
\begin{proof} We normalize our domain to have $\mbox{inrad}(\Omega)=1$ and $\mbox{diam}(\Omega) = N$. Then, for a typical point $x$, we expect $V(x,1) \sim 1$ and $V(y, \delta) \sim \delta^2$ and we obtain the desired estimate. The only way the estimate can fail is if one term is a lot smaller than we expect it to be.
 We start with the case $V(x,1) \leq \varepsilon \leq \mbox{inrad}(\Omega)/1000$. We distinguish two cases: one where the ball of radius 1 around $x$ contains a connected segment of the boundary and one where it contains a disconnected part of the segment (see Fig. 4). In both cases, we can see that the domain $\Omega$ around $x$ can be sandwiched by two lines at distance $\sim \varepsilon$. 
We see, in both cases, that if $V(x,1) = \varepsilon$, then $V(y,\delta) \gtrsim \delta \varepsilon$ (if $\delta \gtrsim \varepsilon$) and $V(y, \delta) \gtrsim \delta^2$ (if $\delta \ll \varepsilon$) and we establish the desired result in both cases. The other case, $V(y,\delta) \leq \varepsilon \delta^2$ is completely analogous.
\end{proof}

With these two ingredients in place, we can obtain one final estimate that will be used in the proof of the result.

\begin{lemma} There exists a universal $\delta_0 > 0$ such that for all $0 < \delta \leq \delta_0$ there exists a constant $ 1 < c_{\delta} < \infty$ depending only on $\delta$ such that for all bounded, convex $\Omega \subset \mathbb{R}^2$ and all $x,y \in \Omega$ with $\|x-y\| \leq 1$
$$ \forall~z \in \Omega: p_{\delta}(x, z) \leq c_{\delta} \cdot p_{1}(y, z).$$
\end{lemma}
\begin{proof}
This follows immediately from Theorem 1 and Lemma 2: we have
$$ p_{\delta}(x, z) \leq  \frac{c_3}{V(x, \sqrt{\delta})} e^{-c_4 \frac{\|x-z\|^2}{\delta}}$$
and 
$$  \frac{c_1}{V(y, 1)} e^{-c_2 \|z-y\|^2} \leq p_1(y,z).$$
By choosing $\delta$ such that $c_4/\delta \geq 2c_2$, we can force faster exponentially decay of $p_{\delta}$ for $z$ far away from $\left\{x,y\right\}$: more precisely, a simple computation shows that then
$$  e^{-c_4 \frac{\|x-z\|^2}{\delta}} \leq  e^{-c_2 \|z-y\|^2}$$
as soon as $\|x-z\| \geq 4$. 
 However, all remaining terms are bounded by universal constants (this is the content of Lemma 2) can thus be controlled by a universal constant $c_{\delta}$.
\end{proof}
We quickly explain the idea behind the estimate: we like to think of heat kernels as localized Gaussians and of $x$ and $y$ as neighboring points (at distance $1 \ll \mbox{diam}(\Omega)$). One of the heat kernels is fixed at time $t=1$. By considering the heat kernel at the other point for shorter time $\delta_0 < 1$, we obtain a more localized heat kernel that has faster decay far away (but also has a larger $L^{\infty}-$norm). By making the constant $c_{\delta}$ sufficiently large, we can dominate the larger localized bump by the wider Gaussian locally; the more localized bump comes with faster decay and is thus automatically controlled far away. The Lemma states that this can be done in a uniform manner.

\subsection{Probabilistic ingredients.}
The Feynman-Kac formula then allows us to interpret $\phi_1$ as being connected to Brownian motion. More precisely, let $B_x(t)$ denote a Brownian motion started in $x$ and running up to time $t$ (being reflected on the boundary of $\Omega$), then 
$$ \mathbb{E} \left(\phi_1( B_x(t))\right) = e^{-\mu_1 t} \phi_1(x).$$
We will start by using the probabilistic interpretation to prove our key Lemma: it says that if the maximum is assumed in a point $x_{\max}$, then the function is very close to the maximum in a $1-$neighborhood by which we mean
$$ \phi_1\big|_{\mbox{\tiny 1-neighborhood of}~x_{\max}} \geq  \left(1 - c \left( \frac{\mbox{inrad}(\Omega)}{\mbox{diam}(\Omega)}\right)^2\right)  \|\phi_1\|_{L^{\infty}(\Omega)}$$
for some universal constant $c$. Indeed, we will actually prove a slightly stronger statement (which implies this particular geometric statement via $\mu_1 \leq c N^{-2}$, see \cite{grieser, jerison}, for some universal constant $c$).

\begin{lemma} For $N$ sufficiently large (depending only on universal constants), let $\Omega$ be a convex domain of dimensions roughly $N \times 1$ (in the sense above). Assume $\phi_1$ assumes its maximum in $x_{\max}$. Then, for some universal $c$, and all $\|x-x_{\max}\| \leq 1$,
$$ \phi_1(x) \geq e^{-c \cdot \mu_1} \phi_1(x_{\max}) .$$
\end{lemma}
\begin{proof} 
We argue by contradiction; let $C \gg 1$ be a very large constant and assume that, for some $\|x - x_{\max}\| \leq 1$,
$$ \phi_1(x) \leq e^{-C \mu_1}\phi_1(x_{\max}).$$
The probabilistic interpretation at time $t=1$ implies
\begin{align*}
  \mathbb{E}(\phi_1(B_{x}(1))) &= e^{-\mu_1} \phi_1(x) \leq \phi_1(x) \leq  e^{-C \mu_1} \phi_1(x_{\max})
    \end{align*}
This tells us something interesting: not only is $\phi_1(x)$ smaller than we expected, even a sort of local average over a neighborhood of $x$ is smaller than we expected. The same computation carried out in the maximum shows 
  $$   \mathbb{E}(\phi_1(B_{x_{\max}}(1))) = e^{-\mu_1 } \phi_1(x_{\max}).$$
There is some decay ( $e^{-\mu_1 }$) but not quite as much as in the other bound (where it is $e^{-C \mu_1}$ and $C \gg 1$).
We will now invoke Lemma 3 in the two points $x_{\max}$ and $x$. There is a universal constant $\delta$ (together with another universal constant $1 \leq c_{\delta} < \infty$) such that
$$ p_{\delta}(x, z) \leq c_{\delta} \cdot p_1(x_{\max}, z).$$
However, this gives rise to an interesting bound on $\phi_1(x)$ from above. Instead of running Brownian motion, we flip a biased coin: with likelihood $c_{\delta}^{-1}$ we run Brownian motion up to time $\delta$, with likelihood $1-c_{\delta}^{-1}$ we simply return the value $\phi_1(x_{\max})$. The expected value of this game is
$$ \mbox{expected value} = c_{\delta}^{-1} e^{-\mu_1 \delta} \phi_1(x) + (1-c_{\delta}^{-1}) \phi_1(x_{\max}).$$ 
At the same time, the inequality between the heat kernels imply that the value of this game is bigger than the expected value of simply running Brownian motion up to time $t=1$ starting in $x_{\max}$ (which is $e^{-\mu_1}\phi_1(x_{\max})$). This shows
$$  c_{\delta}^{-1} e^{-\mu_1 \delta} \phi_1(x) + (1-c_{\delta}^{-1}) \phi_1(x_{\max}) \geq e^{-\mu_1}\phi_1(x_{\max}).$$
We now bound this
\begin{align*}
e^{-\mu_1}\phi_1(x_{\max}) &\leq c_{\delta}^{-1} e^{-\mu_1 \delta} \phi_1(x) + (1-c_{\delta}^{-1}) \phi_1(x_{\max})\\
&\leq c_{\delta}^{-1} e^{-\mu_1 \delta} e^{-C \mu_1} \phi_1(x_{\max}) + (1-c_{\delta}^{-1}) \phi_1(x_{\max})\\
&\leq c_{\delta}^{-1} e^{-C \mu_1} \phi_1(x_{\max}) + (1-c_{\delta}^{-1}) \phi_1(x_{\max})\\
&= \left(1 - c_{\delta}^{-1}(1- e^{-C \mu_1}) \right) \phi_1(x_{\max}).
\end{align*}
For $N$ sufficiently large, the term $C \mu_1$ is rather close to 0 since it decays like $N^{-2}$. 
Taylor expansion shows that we can replace the exponential distribution by its Taylor expansion and we obtain
$$ \left(1 - c_{\delta}^{-1}(1- e^{-C \mu_1}) \right) \sim 1 - c_{\delta}^{-1} C \mu_1 \sim e^{-c_{\delta}^{-1} C \mu_1}$$
which then leads to a contradiction for $C \geq 10 c_{\delta}$ and $N$ sufficiently large (depending on all these universal constants).
\end{proof}

As mentioned above, it seems conceivable that our proof of Lemma 4 (which is the only Lemma that we use in the proof of the main Theorem) can be somewhat simplified so we do not have to appeal to the full strength of Gaussian estimates. One possibility could be the following: we run the heat equation in $x$ for time $t=\delta < 1$. However, there is a nontrivial chance that a Brownian motion started in $x_{\max}$ is close to $x$ at time $t=1-\delta$ and then emulates a Brownian particle started there and we inherit decay estimates. Perhaps there are also purely elliptic estimates that one could use (say, of the type commonly used in proofs of the Gaussian bounds or Harnack estimates).

\subsection{Proof of the Theorem.}
\begin{proof}[Proof of the Theorem] We start by assuming that the domain is oriented as above and that the maximum is assumed in the point $x_{\max}$  (see Fig. 5 for a sketch of what things roughly look like).

\begin{figure}[h!]
\begin{tikzpicture}
\draw[ultra thick] (0,0) -- (7,0);
\draw[ultra thick] (0,3) -- (7,3);
\node at (3.5, -0.5) {$\partial \Omega$};
\filldraw (2, 1) circle (0.06cm);
\node at (2.5, 1) {$x_{\max}$};
\node at (7.5, 3) {$\partial \Omega$};
\draw [->, thick] (0, 1.5) -- (-2, 1.5);
\node at (-1, 1) {to the nodal line};
\draw [thick] (2,0) -- (2,3);
\draw [thick] (6,0) -- (6,3);
\node at (2, -0.5) {$x_*$};
\node at (6, -0.5) {$x_*+c_2$};
\end{tikzpicture}
\caption{Two cross-sections: we assume the maximum is attained in the $x_*$ fiber and show that the values at the fiber $x_{*} + c_2$ are bigger for some universal $c_2$.}
\end{figure}
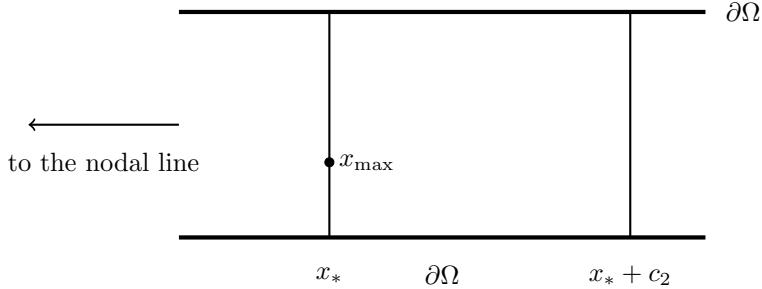

 We recall that Lemma 4 implies that for all $x$ on the fiber going through $x_{\max}$ 
$$ \phi_1(x) \geq e^{-c_1 \mu_1} \phi_1(x_{\max}),$$
where $c_1$ is a universal constant. We want to show that for $c_2$ sufficiently large, $\phi_1$ on the fiber $x_* + c_2$ is larger than it is on the entire fiber $x_*$ (indeed, we will show the strongest possible form of the statement: the smallest value assumed by $\phi_1$ on $x_* + c_2$ is larger than the largest value assumed on $x_*$); this is, of course, a contradiction.
Let now $y$ be an arbitrary point on the $x_* + c_2$ fiber. We can write
$$ \phi_1(y) = e^{\mu_1 t} \cdot \mathbb{E}(\phi_1(B_y(t)))$$
and consider now, path-wise for each such Brownian particle, the first hitting time $T$ of the fiber $x_*$. We split, for each $t$, the Brownian paths into two disjoint sets
$$ A = \left\{B_y(t)\big| ~T \leq t \right\} \qquad \mbox{and} \qquad B ~ \mbox{the complement of}~A.$$
We first show that $B$ can be neglected since, for $t$ large, it is a very small set: we observe 
\begin{align*}
e^{-\mu_1 t}\phi_1(y) &= \mathbb{E}(\phi_1(B_y(t))) \\
&= \mathbb{E}(\phi_1(B_y(t)) \big| A)\mathbb{P}(A) + \mathbb{E}(\phi_1(B_y(t)) \big| B)\mathbb{P}(B) \\
&\underbrace{\geq}_{(*)} \mathbb{E}(\phi_1(B_y(t)) \big| B)\mathbb{P}(B) \\
&\geq \mathbb{P}(B)\min_{z ~{\tiny \mbox{right of}~x_*}} \phi_1(z),
\end{align*}
where $(*)$ follows from the same argument that we will use below (and which will be explained there).
We know (see \cite{jerison}) that the nodal line is simple and thus the minimum is positive and $\mathbb{P}(B)$ is exponentially decaying in $t$ (this step uses, implicitly, the fact that the nodal line and the fiber on which the maximum is attained are not close to one another; this is known, see \cite{jerison}, to a much greater degree of precision than we require here; one could re-derive a weaker but still sufficient statement using similar arguments as the ones we employ here, see \cite{manas}).

This implies
$$ e^{-\mu_1 t}\phi_1(y) \geq  \mathbb{E}(\phi_1(B_y(t)) \big| A)\mathbb{P}(A) $$
and also shows that, for $t$ large, the inequality is not too lossy since $\mathbb{P}(A)$ converges exponentially quickly to 1 (we
will later let $t \rightarrow \infty$ and can thus absorb any constants arising from this step). For any Brownian motion in $A$, we can use Markovianity to write (abusing notation and using $+$ for concatenation of paths)
$$ B_y(t) = B_y(T) + B_{B_y(T)}(t-T).$$
This is the crucial step: it allows us to write the Brownian motion started on the fiber $x_* + c_2$ as a Brownian motion started in the fiber $x_*$ at a later time. This shows, for paths in $A$,
\begin{align*}
 \phi_1(y) &=  e^{\mu_1 t} \cdot \mathbb{E}(\phi_1(B_y(t)))\\
&= \mathbb{E} e^{\mu_1 t} \cdot \phi_1(B_y(t))\\
 & = \mathbb{E} e^{\mu_1 T} e^{\mu_1 (t-T)}  \phi_1(B_y(t))\\
 &=  \mathbb{E} e^{\mu_1 T} e^{\mu_1 (t-T)}  \phi_1(B_{B_T}(t-T))\\
 &= \mathbb{E}  e^{\mu_1 T} \phi_1(B_{B_T}(t-T))
 \end{align*}
This is where our Lemma comes into play: $\phi_1$ is uniformly large on the entire fiber $x_*$ (containing $B_T$) and thus
$$  \phi_1(B_{B_T}(t-T)) \geq e^{-c_1 \cdot \mu_1} \phi_1(x_{\max})$$
and therefore
$$  \phi_1(y) \geq   \mathbb{E} e^{\mu_1 T} e^{-c_1 \mu_1} \phi_1(x_{\max})  =  \phi_1(x_{\max}) \cdot\mathbb{E} ~e^{-c_1 \mu_1} e^{\mu_1 T}.$$
We can now obtain the desired result by showing that, for a sufficient choice of $c_2$ (the distance of the second fiber to the first fiber), we have the inequality
$$ \mathbb{E} ~ e^{\mu_1 T} > e^{c_1 \mu_1}$$
for the hitting times. This, however, is pretty easy: Brownian motion travels, on average, distance $\sim \sqrt{t}$ within $t$ units of time. This shows that by setting $c_2 \sim \sqrt{c_1}$, we obtain the desired statement (the exponential weight is in our favor: half of the Brownian motions require at least $\sim \sqrt{c_1}$ units of time, some of the longer (which is weighted heavier in the exponential distribution). This last step can be made formally precise by appealing again to the decay estimates for the heat kernel from Theorem 1.
\end{proof}

\textit{Remark.} This proof easily carries over to (much) more general settings. The Hot Spots conjecture is known to fail for manifolds (see e.g. Freitas \cite{freitas}), however, every ingredient of our proof carries over to at least a fairly large class of domains: we required a two-sided Gaussian heat kernel estimate but those are known to hold at a rather large level of generality (see \cite{grig, saloff, sturm}). We also needed the ability to partition the domain $\Omega$ by a hypersurface $\Sigma$ into two parts such that: (1) $\phi_1$ on the hypersurface $\Sigma$ is very close to its maximum value and (2) the domain $A$ contains the nodal line, in particular $\phi_1\big|_B > 0$. However, there are certainly other families of domains where these things are possible and, moreover, the argument even holds on finite combinatorial graphs equipped with the Graph Laplacian (though it is substantially harder to come up with natural conditions on the Graph, this could be potentially interesting future work).

\section{Remarks on the Hot Spots conjecture}
We conclude with a series of remarks regarding the hot spots conjecture and possible refinements of our approach. There seem to be two distinct
cases: the case $\mbox{inrad}(\Omega) \sim \mbox{diam}(\Omega)$ and $\mbox{inrad}(\Omega) \ll \mbox{diam}(\Omega)$. The first case seems to have some additional challenges (the first eigenvalue may not be simple), we restrict ourselves to the second case. A typical case that we need to exclude is then shown in Fig. 6.

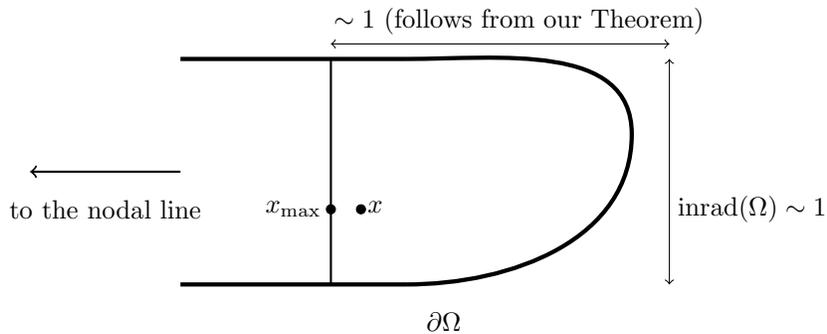
\begin{figure}[h!]
\begin{tikzpicture}
\draw[ultra thick] (0,0) -- (3,0);
\draw[ultra thick] (0,3) -- (3,3);
\draw[ultra thick] (3,3) to[out=0, in=90] (6,2) to[out=270,in=0] (3,0);
\node at (3.5, -0.5) {$\partial \Omega$};
\filldraw (2, 1) circle (0.06cm);
\node at (1.5, 1) {$x_{\max}$};
\filldraw (2.4, 1) circle (0.06cm);
\node at (2.6, 1) {$x_{}$};
\draw [->, thick] (0, 1.5) -- (-2, 1.5);
\node at (-1, 1) {to the nodal line};
\draw [thick] (2,0) -- (2,3);
\draw [<->] (6.5,0) -- (6.5,3);
\node at (7.6, 1) {$\mbox{inrad}(\Omega) \sim 1$};
\draw [<->] (2,3.2) -- (6.5,3.2);
\node at (4.5, 3.5) {$\sim 1$ (follows from our Theorem)};
\end{tikzpicture}
\caption{A sketch of the generic setting for elongated convex domains.}
\end{figure}

It is the natural to assume that, for a suitable chosen $x$ very close to $x_{\max}$ one could hope to obtain
$$ \phi_1(x) - \phi_1(x_{\max}) \geq c\|x-x_{\max}\|$$
which would then be the desired contradiction. Let us suppose $\|x-x_{\max}\| = \varepsilon$ and $\varepsilon \rightarrow 0$. Emulating our argument, we will interpret the line through $x_{\max}$ as a stopping barrier for Brownian motion and reduce the problem to studying
$$ \mbox{the size of} \qquad \mathbb{E} ~\phi_1(B_x(T)) e^{\mu_1 T}.$$
Standard estimates suggest a dichotomy: roughly $\mbox{const} \cdot \varepsilon$ of Brownian particles `escape' (in the sense of not hitting the wall but hitting the boundary $\partial \Omega$ instead, roughly $1- \mbox{const} \cdot \varepsilon$ hit the wall before hitting the boundary. The ones hitting the wall can be analyzed fairly completely, there is a fairly straightforward computation exploiting the reflection principle that is, for example, carried out in \cite{jianfeng}. However, the typical hitting time for these particles is rather small and on the scale $\varepsilon^2 \lesssim T \ll \varepsilon$, this means that the factor $e^{\mu_1 T}$ does not contribute very much. However, there are also other particles, the ones hitting the boundary first. While their portion is small (only $\varepsilon$ of particles), their effect is hopefully larger than that because $T \sim_{x_{\max}, \Omega} 1$. 
In that case, we see that $e^{T \mu_1} \sim_{x_{\max}, \Omega} 1 + c N^{-2}$. The remaining question is now: where does the particle impact on the line? Presumably, due to convexity, it is slightly more likely to impact close to $x_{\max}$ than it is to impact far away (where $\phi_1$ would be smaller) -- a quantification of this is what we consider the missing isoperimetric ingredient. We mention an inequality of Ba\~nuelos \& Pang \cite{ban} as an example of the flavor of what such an inequality could look like. \\

\textbf{Acknowledgment.} The author is grateful to Laurent Saloff-Coste for helpful discussions.

\end{document}